\documentclass[11pt,a4paper]{amsart}
\usepackage[margin=1in]{geometry}
\usepackage{amsmath,amssymb,amsthm,mathtools}
\usepackage{microtype}
\usepackage{enumitem}
\usepackage[hidelinks]{hyperref}
\usepackage{needspace}
\usepackage[T1]{fontenc}
\usepackage{lmodern}
\allowdisplaybreaks[2]
\newcommand{\R}{\mathbb{R}}
\newcommand{\Gal}{\operatorname{Gal}}
\newcommand{\Tr}{\operatorname{Tr}}
\newcommand{\diag}{\operatorname{diag}}
\newcommand{\GL}{\operatorname{GL}}

\theoremstyle{plain}
\newtheorem{theorem}{Theorem}[section]
\newtheorem{proposition}[theorem]{Proposition}
\newtheorem{lemma}[theorem]{Lemma}
\newtheorem{corollary}[theorem]{Corollary}
\theoremstyle{definition}

\numberwithin{equation}{section}
\newcommand{\Q}{\mathbb{Q}}
\newcommand{\C}{\mathbb{C}}
\newcommand{\Z}{\mathbb{Z}}
\newcommand{\Pone}{\mathbb{P}^1}
\newcommand{\cE}{\mathcal{E}}
\title[Algebraically primitive Teichm\"uller curves in $\Omega\mathcal{M}_g(g-1,g-1)^{\mathrm{hyp}}$]{Complete classification of algebraically primitive Teichm\"uller curves in $\Omega\mathcal{M}_g(g-1,g-1)^{\mathrm{hyp}}$}
\author{Myeongjae Lee}
\address{Department of Mathematics\\
Indiana University\\
831 E.\ Third Street\\
Bloomington, IN 47405\\
USA}
\email{myelee@iu.edu}

\date{}
\begin{document}
\begin{abstract}
We give a complete classification of algebraically primitive Teichm\"uller curves in $\Omega\mathcal M_g(g-1,g-1)^{\mathrm{hyp}}$ for $g>2$. These curves are precisely those generated by the Veech $2p$-gons, where $p=2g+1$ is prime.
\end{abstract}
\maketitle

\tableofcontents

\section{Introduction} \label{sec:intro}
A Teichm\"uller curve \(C\) is an algebraic curve in \(\mathcal{M}_g\) that is
isometrically immersed for the Teichm\"uller metric. We consider curves
generated by abelian differentials, namely projections of closed
\(\mathrm{GL}^+(2,\mathbb{R})\)-orbits of translation surfaces
\((X,\omega)\). A curve is said to lie in a stratum component when its
generating differential lies there. The trace field \(K\) of \(C\) is defined as
\[
    K\coloneqq \mathbb{Q}(\{\operatorname{tr}(A):A\in \operatorname{SL}(X,\omega)\})\subset \mathbb{R}
\]
where \(\operatorname{SL}(X,\omega)\) is the Veech group of the generating
translation surface \((X,\omega)\). The trace field \(K\) has degree at most
\(g\) over \(\mathbb{Q}\). The Teichm\"uller curve \(C\) is called
\emph{algebraically primitive} if \([K:\mathbb{Q}]=g\).
In $g=2$, the algebraically primitive curves are completely classified:
the irreducible components of the Weierstrass curves \(W_D\subset\mathcal{M}_2\) for nonsquare discriminants $D$
\cite{McM03,Cal04,McM05}, and a unique curve in \(\mathcal{H}(1,1)\),
generated by the regular decagon \cite{McM06}. For $g>2$, Eskin--Filip--Wright \cite{EFW} proved that there are only
finitely many Teichm\"uller curves in each genus whose trace field has
degree greater than two. Earlier finiteness results in the present
setting and related strata include the following. M\"oller \cite{Mol08} proved finiteness of algebraically primitive curves in the hyperelliptic component
\(\Omega\mathcal{M}_g(g-1,g-1)^{\mathrm{hyp}}\). Bainbridge--M\"oller \cite{BM} obtained finiteness in \(\Omega\mathcal{M}_3(3,1)\), and
Bainbridge--Habegger--M\"oller \cite{BHM} established finiteness in any stratum of genus three. In the minimal stratum in prime genus,
Matheus--Wright \cite{MW} proved finiteness in \(\Omega\mathcal{M}_g(2g-2)\) for prime \(g> 2\), and Nguyen--Wright \cite{NW} together with Matheus--Wright \cite{MW} obtained finiteness of non-arithmetic Teichm\"uller curves in \(\Omega\mathcal{M}_3(4)^{\mathrm{hyp}}\). Winsor \cite{Winsor} gave a complete classification in \(\Omega\mathcal{M}_3(2,2)^{\mathrm{hyp}}\), where the Veech \(14\)-gon generates the unique algebraically primitive Teichm\"uller curve.
In this paper, we use the description of the cusps by M\"oller \cite{Mol08}, together with the Harder--Narasimhan filtration of the Hodge bundle given by Yu--Zuo \cite{YZ}, as it is applied by Bainbridge--Habegger--M\"oller in \cite{BHM}. We also use the fact that cylinder circumferences and heights form trace-dual bases, stated in \cite{BM,Winsor}. As a result, we have the complete classification in \(\Omega\mathcal{M}_g(g-1,g-1)^{\mathrm{hyp}}\) for any $g\ge 3$, generalizing \cite{McM06} and \cite{Winsor}.

\begin{theorem}\label{thm:classification}
The component
$\Omega\mathcal M_g(g-1,g-1)^{\mathrm{hyp}}$ for $g\ge 3$ contains an algebraically primitive Teichm\"uller curve if and only if $p=2g+1$ is prime.
In that case, the unique such curve is generated by the Veech $2p$-gon.
\end{theorem}

\subsection{Strategy of proof}

\begin{itemize}
    \item In Section~\ref{sec:intro}, we describe the normalized irreducible cusps of the algebraically primitive curves.
    \item In Section~\ref{sec:HN-duality}, we use the Harder--Narasimhan filtration and trace duality to obtain the quotient polynomials and sine orthogonality.
    \item In Section~\ref{sec:degree}, we prove that $[L:\Q]=\varphi(N)=2g$ or $4g$ and compute $\Gal(L/K)$.
    \item In Section~\ref{sec:reduction}, we prove the necessity in Theorem~\ref{thm:classification}: $p=2g+1$ is prime. Moreover, we have $N=p$ or $2p$, and the node preimages are all primitive $N$-th roots.
    \item In Section~\ref{sec:uniqueness}, we prove uniqueness of the Veech $2p$-gon, the second statement of Theorem~\ref{thm:classification}. 
\end{itemize}

\paragraph{\bf Acknowledgement}
The author would like to thank Matt Bainbridge for many valuable discussions and for introducing Yu--Zuo's result \cite{YZ} on the Harder--Narasimhan filtration of the Hodge bundle to the author. 

\paragraph{\bf AI Usage Disclosure}
The author acknowledges the use of ChatGPT-5.6 to assist with proving Theorem~\ref{thm:degree} in an earlier version, which later helped the author realize that we have a much stronger Lemma~\ref{lem:top-duality} which can lead to the complete classification result. All core mathematical reasoning and proof construction in the current version were performed and validated by the author.

\subsection{Cusps of Teichm\"uller curves}\label{sec:cusps}
Let \(C\subset \Omega\mathcal{M}_g(g-1,g-1)^{\mathrm{hyp}}\) be an
algebraically primitive Teichm\"uller curve, with totally real trace field
\(K\) and \([K:\Q]=g\). Throughout this paper we assume \(g\ge 3\).
By M\"oller \cite{Mol08}, \(K\) is a subfield of a cyclotomic field, hence abelian and in particular Galois over \(\Q\).
An irreducible periodic direction exists: choose a saddle connection
joining the two zeros and apply Veech's periodicity theorem.  With
only two zeros, its saddle-connection graph is connected; algebraic
primitivity then gives $g$ cylinders and rational normalization
\cite[\S2]{Winsor}.
M\"oller's cusp description \cite{Mol08}, together with the trace-dual
height result \cite[Lemma~10.4]{BM}, \cite[Theorem~2.4]{Winsor}, gives the following.

\begin{proposition}\label{prop:MM} 
Consider an irreducible horizontal cusp whose
normalization is a meromorphic \(1\)-form \(\omega_h^{\mathrm{norm}}\)
on \(\Pone_z\), written in two ways:
\[
\omega_h^{\mathrm{norm}}
=\sum_{i=1}^g c_i\Bigl(\frac{1}{z-x_i}-\frac{1}{z-x_i^{-1}}\Bigr)\,dz
=R\frac{z^{g-1}}{\prod_{i=1}^g (z-x_i)(z-x_i^{-1})}\,dz,\]
where the $2g$ node preimages $x_i^{\pm1}$ are distinct and different
from $\pm1$, the normalized residues $c_i$ are positive real, $R\in i\R^\times$,
and the marked zeros of
\(\omega_h^{\mathrm{norm}}\) are at \(z=0,\infty\). Then
\begin{itemize}
    \item The \(x_i\) are roots of unity;
    \item \(K\subset L\coloneqq \mathbb{Q}(x_1,\dots,x_g)\);
    \item A real multiple of \(\{c_i\}\) forms a \(\Q\)-basis of \(K\).
    \item Let $(h_i)$ be the
    trace-dual basis of $(c_i)$.  Then
    \[
       \frac{h_i c_k}{c_i h_k}\in\Q_{>0}\qquad1\le i,k\le g.
    \]
    For a generating differential $(X,\omega)$ near the cusp, the cylinder circumferences and heights are common positive real multiples of $(c_i)$ and $(h_i)$, respectively.
\end{itemize}
\end{proposition}

Note that \(L=\mathbb{Q}(\zeta_N)\) where
\(N\coloneqq \operatorname{lcm}_i(\operatorname{ord}x_i)\). We may assume \(c_i\in K\)
and \(K=\mathbb{Q}(c_1,\dots,c_g)\) by a common real scaling of
\(\omega_h^{\mathrm{norm}}\) if necessary. Write
\[
 x_i=\zeta_N^{n_i},\qquad \gcd(n_1,\ldots,n_g,N)=1.
\]
The gcd assertion follows from the definition of $N$: otherwise all node orders
would divide a proper divisor of $N$.
We identify $\Gal(L/\Q)$ with $(\Z/N\Z)^\times$ by
\[
 \tau_u(\zeta_N)=\zeta_N^u.
\]
Every $\tau_u$ preserves $K$, since $L/\Q$ is abelian. We reserve $\tau_u$ for
these power automorphisms. The notation introduced above will be used throughout this paper. Furthermore, we set
\[
 y_i=x_i+x_i^{-1},\qquad d_i=x_i-x_i^{-1},\qquad
 P(Y)=\prod_{i=1}^g(Y-y_i).
\]
The $y_i$ are pairwise distinct, and $d_i\ne0$ for every $i$.

Then we have the relation 
\[
    \omega_h^{\mathrm{norm}}=\frac{R dz}{zP(z+z^{-1})}
\] and
\begin{equation*}
    c_i=\operatorname{res}_{x_i}\omega_h^{\mathrm{norm}} = \frac{R}{d_iP'(y_i)}.
\end{equation*}
The zero of order $g-1$ at $0$ also gives
\begin{equation}\label{eq:moments}
 \sum_{i=1}^g c_i(x_i^r-x_i^{-r})=0,\qquad 1\le r\le g-1.
\end{equation}
Indeed, expanding the partial fractions at $0$ gives
\[
 \frac{\omega_h^{\mathrm{norm}}}{dz}
 =\sum_{r\ge1}\left(\sum_{i=1}^g c_i(x_i^r-x_i^{-r})\right)z^{r-1},
\]
and its coefficients of $z^0,\ldots,z^{g-2}$ vanish.

\section{Consequences of Harder--Narasimhan filtration and Trace duality}\label{sec:HN-duality}

\subsection{Harder--Narasimhan filtration}
Let $f:\mathcal S\to\overline C$ be the semistable family of curves (after a finite
base change and compactification), with the two zeros extended as
sections $D_0,D_\infty$.  The extended Hodge bundle
$\cE=f_*\omega_{\mathcal S/\overline C}$ splits into eigenline bundles
$\mathcal L_\sigma$ indexed by the embeddings of $K$.
The tautological line bundle is denoted by $\mathcal L$. Yu--Zuo identify the relevant Weierstrass filtration with the Harder--Narasimhan
filtration of $\cE$ \cite[Proposition~5.5 and Theorem~5.6]{YZ}; the direct-sum
principle in \cite[Lemma~4.4]{BHM} identifies its steps with eigenlines.

\begin{proposition}
\label{prop:HN-twozero}
There is an ordering $\sigma_1=1,\sigma_2,\ldots,\sigma_g\in \operatorname{Gal}(K/\Q)$ such that
\[
 0\subset V_1\subset\cdots\subset V_g=\cE,
 \qquad V_j=\bigoplus_{i=1}^j\mathcal L_{\sigma_i},
 \qquad
 \frac{\deg(V_j/V_{j-1})}{\deg\mathcal L}=1-\frac{j-1}{g}.
\]
The bundle $V_j$ consists of differentials vanishing at least $g-j$
at each marked zero. A nonzero generator of
$\mathcal L_{\sigma_j}$ has exact order $g-j$ at both marked zeros.
\end{proposition}

\begin{proof}
The two marked zeros are exchanged by the hyperelliptic involution.
The Weierstrass flag in this case is
\[
 V_j=f_*\bigl(\omega_{\mathcal S/\overline C}
                 (-(g-j)(D_0+D_\infty))\bigr).
\]
Its successive normalized slopes are given by Yu--Zuo
\cite[Proposition~5.5 and Theorem~5.6]{YZ}, and they are strictly decreasing.

Because the Hodge bundle is a direct sum of eigenlines, the direct-sum principle \cite[Lemma~4.4]{BHM} identifies $V_j$ with the sum of the first $j$ eigenlines as a result of the uniqueness of Harder--Narasihman filtration. 

We check the fiberwise assertion, including the irreducible cusp. On a smooth
fiber $X_t$, the divisor $D_0(t)+D_\infty(t)$ is a fiber of the hyperelliptic
map, and the canonical divisor is equivalent to $(g-1)(D_0(t)+D_\infty(t))$.
It follows that
\[
 h^0\!\left(X_t,\omega_{X_t}(-(g-j)(D_0(t)+D_\infty(t)))\right)=j.
\]
At an irreducible cusp, sections of the dualizing sheaf pull back to
meromorphic differentials on $\Pone_z$ with at most simple poles at the node
branches and opposite residues at each pair. These differentials are
anti-invariant under $z\mapsto z^{-1}$. Dividing by
$dz/(zP(z+z^{-1}))$ gives an invariant rational function with possible poles
only at $0,\infty$. Imposing vanishing at least $g-j$ at both points bounds
these pole orders by $j-1$. Thus the required sections are exactly
\[
 A(z+z^{-1})\frac{dz}{zP(z+z^{-1})},\qquad
 A\in\C[Y],\quad \deg A\le j-1.
\]
This space has dimension $j$. The dimension is therefore constant in a
neighborhood of the cusp. Cohomology and base change identifies the fiber
of $V_j$ with this space, and similarly on the smooth fibers. For each $j\ge1$, a nonzero vector in the fiber of $\mathcal L_{\sigma_j}$ lies
in $V_j$ but not in $V_{j-1}$, by the direct-sum decomposition. The forms in $\mathcal L_{\sigma_j}$ has order exactly $g-j$ at both marked zeroes.
\end{proof}

The first consequence of the Harder--Narasimhan filtration is the following

\begin{corollary}\label{prop:deg}
For every $j$ there is a unique polynomial $Q_j\in\C[Y]$ of degree
exactly $j-1$ such that
\[
                    Q_j(y_i)=\frac{\sigma_j(c_i)}{c_i}.
\]
\end{corollary}
\begin{proof}
As in Section~6 of \cite{BHM}, the eigenform associated with $\sigma_j$ at
this cusp can be normalized as
\[
 \omega_h^j
 =\sum_{i=1}^g\sigma_j(c_i)
 \left(\frac1{z-x_i}-\frac1{z-x_i^{-1}}\right)dz.
\]
Note that the node preimages are unchanged. This is the eigenform associated with
$\sigma_j$.

By Proposition~\ref{prop:HN-twozero}, the form $\omega_h^j$ belongs
to $V_j$ but not to $V_{j-1}$, with $V_0=0$. The description of $V_j$
in that proposition gives
\[
 \omega_h^j=Q_j(z+z^{-1})\omega_h^{\mathrm{norm}},
 \qquad \deg Q_j=j-1.
\]
Taking residues at $x_i$ proves the desired formula. $Q_j$ is uniquely determined by its values at the $g$ distinct points $y_i$.
\end{proof}

\subsection{Trace duality}

By Proposition~\ref{prop:MM}, there exist $\rho\in K^\times$ and
$q_i\in\Q_{>0}$ such that $h_i=\rho q_i c_i$. Set
$D_i=q_i^{-1}$ and $D=\diag(D_i)$.  Trace duality gives
\begin{equation*}
                  \Tr_{K/\Q}(\rho c_i c_j)=D_i\delta_{ij}.
\end{equation*}

\begin{lemma}\label{lem:top-duality}
Let $S=(\sigma_j(c_i))_{i,j}=(c_iQ_j(y_i))_{i,j}$, with rows indexed by nodes
and columns indexed by embeddings, and let
$H=\diag(\sigma_1(\rho),\ldots,\sigma_g(\rho))$.  Then

\[
             SHS^{\mathsf T}=D, \qquad S^{\mathsf T}D^{-1}S=H^{-1}.
\]

In particular, $S^{\mathsf T}D^{-1}S$ is diagonal. As a result, we have

\begin{enumerate}
    \item There exists $B\in L^\times$ such that $\sigma_g(c_i)=BD_i d_i$ for $i=1,\dots,g$.
    \item For any $u,v \in (\Z/N\Z)^\times$ with $\tau_{uv^{-1}}\notin\Gal(L/K)$, we have 
    \[
        \sum_{i=1}^g D_i(x_i^u-x_i^{-u})(x_i^v-x_i^{-v})=0.
    \]
    \item There exist $s\in(\Z/N\Z)^\times$ and $B_0\in L^\times$ with 
    \[
        c_i=B_0D_i(x_i^s-x_i^{-s}), \qquad i=1,\dots, g.
    \]
\end{enumerate}
\end{lemma}

\begin{proof}
The $(i,j)$-entry of $SHS^{\mathsf T}$ is
$\sum_k\sigma_k(\rho c_i c_j)=D_i\delta_{ij}$. So $SHS^{\mathsf T}=D$.
Since $D$ is invertible, $S$ is also invertible. Taking inverses gives
$S^{\mathsf T}D^{-1}S=H^{-1}$.

Let $T$ consist of the first $g-1$ rows of $S^{\mathsf T}$. Then
$\operatorname{rank}T=g-1$. For $j<g$, Proposition~\ref{prop:HN-twozero}
gives $\omega_h^j(0)=0$. Expanding its partial fractions at $0$ gives
\[
 \sum_i\sigma_j(c_i)d_i=0,\qquad j=1,\ldots,g-1.
\]
Thus $(d_i)_i \in \ker T$. The off-diagonal entries in the last
column of $S^{\mathsf T}D^{-1}S$ give
\[
 \sum_i\frac{\sigma_j(c_i)\sigma_g(c_i)}{D_i}=0,
                    \qquad j=1,\ldots,g-1.
\]
Hence also $(\sigma_g(c_i)/D_i)_i \in \ker T$. Both vectors are
nonzero, and $\ker T$ is one-dimensional. Therefore
\[
 \sigma_g(c_i)=BD_i d_i,\qquad i=1,\ldots,g,
\]
where $B=\sigma_g(c_1)/(D_1d_1)\in L^\times$. This proves (1).

Apply $\tau_u,\tau_v\in\Gal(L/\Q)$ to (1). Since $D_i$ is rational,
\[
 \begin{aligned}
 \tau_u(\sigma_g(c_i))&=\tau_u(B)D_i(x_i^u-x_i^{-u}),\\
 \tau_v(\sigma_g(c_i))&=\tau_v(B)D_i(x_i^v-x_i^{-v}).
 \end{aligned}
\]
Both automorphisms preserve $K$. If $\tau_{uv^{-1}}\notin\Gal(L/K)$, their
restrictions to $K$ are distinct, so
$\tau_u|_K\circ\sigma_g$ and $\tau_v|_K\circ\sigma_g$ are distinct
automorphisms of $K$. Thus the columns of $S$ corresponding to $\tau_u|_K\circ\sigma_g$ and $\tau_v|_K\circ\sigma_g$ are orthogonal
for the metric $D^{-1}$. Substituting the preceding formulas and canceling the
nonzero factor $\tau_u(B)\tau_v(B)$ proves (2).

Finally, choose a lift $\tau_s$ of $\sigma_g^{-1}$ to $\Gal(L/\Q)$.
Applying it to (1) gives
\[
 c_i=\tau_s(B)D_i(x_i^s-x_i^{-s}),\qquad i=1,\ldots,g.
\]
This proves (3) with $B_0=\tau_s(B)\in L^\times$.
\end{proof}

\section{The cyclotomic extension \texorpdfstring{$L/\Q$}{L/Q}}\label{sec:degree}

In this section, we study the cyclotomic field $L=\Q(x_i)=\Q(\zeta_N)$.
Denote $L^+$ its maximal real subfield.

\begin{theorem}\label{thm:degree}
Suppose that \(\Omega\mathcal{M}_g(g-1,g-1)^{\mathrm{hyp}}\) has an
algebraically primitive Teichm\"uller curve with trace field $K$ and
$L=\Q(x_1,\ldots,x_g)$. Then $[L:K]=2$ or $4$, thus
$[L:\Q]=2g$ or $4g$. More precisely,
\begin{enumerate}[label=\textup{(\roman*)},leftmargin=*]
\item If $[L:K]=2$ \emph{(Type A)}, then $K=L^+$ and
\[
 \Gal(L/K)=\{\tau_1,\tau_{-1}\}.
\]
\item If $[L:K]=4$ \emph{(Type B)}, then $4\mid N$, every node
exponent $n_i$ is odd, and
\[
 \Gal(L/K)=\{\tau_1,\tau_{-1},\tau_{1+N/2},\tau_{-1-N/2}\}.
\]
\end{enumerate}
For $u,v\in(\Z/N\Z)^\times$, we have
$\tau_{uv^{-1}}\in\Gal(L/K)$ if and only if
$u\equiv\pm v\pmod N$ in Type A, and if and only if
$u\equiv\pm v$ or $u\equiv\pm v+N/2\pmod N$ in Type B.
\end{theorem}
\begin{proof}
By Corollary~\ref{prop:deg}, we have
\[
 \frac{\sigma_2(c_i)}{c_i}=a+by_i,\qquad b\ne0.
\]
Set
\[
 z_i\coloneqq\frac{y_i-y_1}{y_2-y_1}
 =\frac{\sigma_2(c_i)/c_i-\sigma_2(c_1)/c_1}
 {\sigma_2(c_2)/c_2-\sigma_2(c_1)/c_1}\in K.
\]
Then $z_1=0$, $z_2=1$, and $y_i=\alpha+\beta z_i$ with
$\alpha=y_1$, $\beta=y_2-y_1\ne0$. By (1) of
Lemma~\ref{lem:top-duality},
\[
 \frac{d_i}{d_1}
 =\frac{D_1\sigma_g(c_i)}{D_i\sigma_g(c_1)}\in K.
\]
For $\tau\in\Gal(L/K)$, set
\[
 a_\tau=\frac{\tau(\beta)}{\beta},\qquad
 b_\tau=\tau(\alpha)-a_\tau\alpha,\qquad
 c_\tau=\frac{\tau(d_1)}{d_1}.
\]
Then $a_\tau,c_\tau\ne0$ and, for every $i$,
\[
 \tau(y_i)=a_\tau y_i+b_\tau,\qquad
 \tau(d_i)=c_\tau d_i.
\]
Applying $\tau$ to $y_i^2-d_i^2=4$ gives
\[
 (a_\tau^2-c_\tau^2)y_i^2+2a_\tau b_\tau y_i
       +b_\tau^2+4c_\tau^2-4=0.
\]
This polynomial of degree at most two vanishes at the $g\ge3$ distinct
$y_i$, so all its coefficients vanish. Since $a_\tau\ne0$, we obtain
\[
 b_\tau=0,\qquad a_\tau^2=c_\tau^2=1.
\]
Thus $\tau$ sends every $x_i=(y_i+d_i)/2$ by the same one of the four
rules
\[
 x_i\longmapsto x_i,\quad x_i^{-1},\quad -x_i,\quad -x_i^{-1}.
\]
Since the $x_i$ generate $L$, at most four automorphisms fix $K$.
The extension $L/K$ is Galois, and complex conjugation is a nontrivial
automorphism fixing the totally real field $K$. Hence $[L:K]=2$ or $4$.

In Type A, complex conjugation generates $\Gal(L/K)$, so $K=L^+$.
In Type B, all four sign choices occur. In particular, there is an
automorphism $\nu$ fixing $K$ and sending every $x_i$ to $-x_i$.
Thus $-1=\nu(x_i)/x_i$ is an $N$-th root of unity, and $N$ is even.
Let $\nu=\tau_v$. Then $v$ is odd and
\[
 (v-1)n_i\equiv N/2\pmod N.
\]
Since $\gcd(n_1,\ldots,n_g,N)=1$, doubling these congruences gives
$2(v-1)\equiv0\pmod N$. Since $\nu\ne1$, we have
$v\equiv1+N/2\pmod N$. Its odd parity proves $4\mid N$. Substitution
gives $(n_i-1)N/2\equiv0\pmod N$, so every $n_i$ is odd. Complex
conjugation and $\nu$ generate $\Gal(L/K)$, proving the stated kernel.
\end{proof}

These types refer to the chosen cusp, not initially to the whole curve.
In the next section, we prove that no Type B cusp satisfies the preceding
conditions.

\section{Nonexistence for non-prime \texorpdfstring{$2g+1$}{2g+1}}\label{sec:reduction}

In this section, we prove that an algebraically primitive Teichm\"uller curve exists only if $2g+1$ is prime. More precisely, we have the following.

\begin{theorem}\label{thm:cusp-rigidity}
For an algebraically primitive Teichm\"uller curve, $p=2g+1$ is prime,
$N=p$ or $2p$, and $K=\Q(\zeta_p)^+$, the maximal real subfield of $\Q(\zeta_p)$. The node preimages
$\{x_i,x_i^{-1}:1\le i\le g\}$ are exactly the primitive $N$-th roots,
and the weights $D_i$ are all equal. In the
original normalization,
\[
 \omega_h^{\mathrm{norm}}=R\frac{z^{g-1}}{\Phi_N(z)}\,dz,\qquad R\ne0,
\]
where $\Phi_N$ is the $N$-th cyclotomic polynomial. Up to a common nonzero scalar and the coordinate change $z\mapsto-z$, the
stable differential is represented by
\begin{equation}\label{eq:polygon-cusp}
              \omega_p=\frac{z^{g-1}}{\Phi_p(z)}\,dz,
                       \qquad \zeta_p^a\sim\zeta_p^{-a},
\end{equation}
where $\Phi_p(z)=1+z+\dots+z^{p-1}$ is the $p$-th cyclotomic polynomial.
\end{theorem}

\subsection{Fourier coefficients and the weighted node counting}

Define a function
\begin{equation*}
                   F(a)=\sum_{i=1}^g D_i\cos(2\pi n_i a/N).
\end{equation*} for $a\in\Z/N\Z$.
The product-to-sum identity gives
\begin{align}\label{eq:F-equality}
 F(u-v)-F(u+v)
 &=2\sum_iD_i\sin(2\pi n_i u/N)\sin(2\pi n_i v/N)\notag\\
 &=-\frac12\sum_iD_i(x_i^u-x_i^{-u})(x_i^v-x_i^{-v}).
\end{align} 
By (2) of Lemma~\ref{lem:top-duality}, we have $F(u-v)=F(u+v)$ if $\tau_{uv^{-1}}\notin\Gal(L/K)$. 

We can express $F(a)$ in a different way:
\begin{equation*}
 \mu_\ell=\frac12\sum_{i=1}^gD_i
       \bigl(\delta_{n_i,\ell}+\delta_{-n_i,\ell}\bigr),
 \qquad
 F(a)=\sum_{\ell\bmod N}\mu_\ell e^{2\pi i a\ell/N}.
\end{equation*}
This coefficient records the locations of the pair of nodes $x_i,x_i^{-1}$ 
with their weights $D_i$. 

If $N$ is odd, set
\begin{equation*}
       \lambda_\ell\coloneqq \mu_\ell,\qquad G(a)=F(a),
                    \qquad \ell,a\in\Z/N\Z.
\end{equation*}
If $N$ is even, the units $u,v\in(\Z/N\Z)^\times$ are odd, so
\eqref{eq:F-equality} compares only $F$ at even residue classes. These classes
cannot distinguish the positions $\ell$ and $\ell+N/2$, since
\[
 e^{2\pi i a(\ell+N/2)/N}=e^{2\pi i a\ell/N}
\] for even $a$. So we set
\begin{equation*}
 \lambda_\ell=\mu_\ell+\mu_{\ell+N/2},
               \qquad G(a)=F(2a).
\end{equation*} for $\ell,a\in \Z/(N/2)\Z$. 
Then
\begin{equation*}
            G(a)=\sum_{\ell\bmod N/2}\lambda_\ell
                                             e^{4\pi i a\ell/N}.
\end{equation*}
Now $\lambda_\ell$ records the node pairs chosen from $\{\zeta_N^{\pm\ell}, -\zeta_N^{\pm\ell}\}$. The coefficients $\lambda_{\ell}$ always satisfy:
\begin{equation*}
 \lambda_\ell\ge0,\qquad \lambda_{-\ell}=\lambda_\ell,
 \qquad \lambda_0=0,
 \qquad \Lambda:=G(0)=\sum_iD_i>0.
\end{equation*}
The vanishing of $\lambda_0$ follows from $x_i\ne\pm1$: for odd $N$,
$\lambda_0=\mu_0$, while for even $N$, $\lambda_0=\mu_0+\mu_{N/2}$.

In Type B, we also have
\begin{equation*}
 \lambda_\ell=0\quad\text{for }\ell\text{ even}
\end{equation*} 
because every node exponent $n_i$ is odd by Theorem~\ref{thm:degree},
and $N/2$ is even. Moreover,
\begin{equation}\label{eq:B-antiperiod}
 G(a+N/4)=\sum_iD_i(-1)^{n_i}\cos(4\pi a n_i/N)=-G(a).
\end{equation}

Each node pair contributes to at most one nonzero inversion class. Consequently
\begin{equation*}
       \#\bigl(\{\ell:\lambda_\ell>0\}/\{\pm1\}\bigr)\le g.
\end{equation*}

The Fourier inversion formula gives 
\begin{equation*}
 \lambda_\ell=
 \begin{cases}
 \displaystyle\frac1N\sum_{a\bmod N}G(a)e^{-2\pi i a\ell/N},
                                                   &N\text{ odd},\\[6pt]
 \displaystyle\frac2N\sum_{a\bmod N/2}G(a)e^{-4\pi i a\ell/N},
                                                   &N\text{ even}.
 \end{cases}
\end{equation*}

\subsection{Computation of Fourier coefficients}\label{subsec:direct-comparison}

Now we compare values of $G$. The comparison obtained
from \eqref{eq:F-equality} excludes $a=0$ in Type A and $a=0,N/4$ in
Type B. We call every other residue class \emph{admissible}. Recall that
$a\in\Z/N\Z$ if $N$ is odd and $a\in\Z/(N/2)\Z$ if $N$ is even.

\begin{lemma}\label{lem:comparison}
Let $a,b$ be admissible. Then $G(a)=G(b)$ whenever
\begin{equation}\label{eq:comparison-condition}
 b\not\equiv\pm a\pmod p
       \quad\text{for every odd prime }p\mid N,
\end{equation}
and, if $4\mid N$, $a$ and $b$ have opposite parity.

\end{lemma}
\begin{proof}
For odd $N$, take $u=(a+b)/2$ and $v=(b-a)/2$ modulo $N$.
Condition~\eqref{eq:comparison-condition} makes both $u,v$ units.
They satisfy $u-v=a$ and $u+v=b$.  Since $a,b\ne0$, one has
$u\ne\pm v$, so \eqref{eq:F-equality} gives $G(a)=G(b)$.

For even $N$, choose lifts of $a,b$ modulo $N$ and set
$u=a+b$, $v=b-a$.  Then $u-v=2a$ and $u+v=2b$.
If $N\equiv2\pmod4$, adding $N/2$ to one lift changes its parity,
so the lifts may be chosen with opposite parity.  If $4\mid N$,
this is the stated parity hypothesis.  Thus $u,v$ are odd, and
\eqref{eq:comparison-condition} makes them units at every odd prime.
The coincidences $u\equiv\pm v$ would give $a=0$ or $b=0$ modulo
$N/2$.  In Type B the additional coincidences
$u\equiv\pm v+N/2$ would give $a=N/4$ or $b=N/4$.
Admissibility and Theorem~\ref{thm:degree} exclude precisely these
cases.  Equation~\eqref{eq:F-equality} again gives $G(a)=G(b)$.

\end{proof}

\begin{lemma}\label{lem:constancy} We can compute the function $G$ as follows:
\begin{enumerate}[label=\textup{(\roman*)},leftmargin=*]
\item In Type A, $G(a)=C$ for every $a\ne0$ if $12\nmid N$.
If $12\mid N$, then $G(a)=C+Ef(a)$ for every $a\ne0$.
\item In Type B, for every $a\ne0,N/4$,
\[
 G(a)=
 \begin{cases}
  0,&3\nmid N\ \text{or}\ 24\mid N,\\
  Ef(a),&N\equiv12\pmod{24}.
 \end{cases}
\]
\end{enumerate}
Here $C,E$ are real constants; in the cases $12\mid N$, the function
$f$ is well-defined modulo $N/2$ by
\begin{equation*}
 f(a)=
 \begin{cases}
  1,&a\equiv0,1,5\pmod6,\\
 -1,&a\equiv2,3,4\pmod6.
 \end{cases}
\end{equation*}
The value $G(0)=\Lambda$ is treated separately.
\end{lemma}
\begin{proof}
At every odd prime $q\ge5$ dividing $N$, there is a residue $c_q\pmod q$
avoiding $\pm a,\pm b$, since at most four residues are forbidden.
For comparison with only $a$, at most two residues are forbidden. Lift
these choices modulo the corresponding prime powers. The Chinese
remainder theorem combines them with the required conditions at $2,3$.
If the resulting $c$ is admissible and satisfies the comparison condition
with both $a,b$, Lemma~\ref{lem:comparison} gives
\[
                         G(a)=G(c)=G(b).
\]
We give the choices ensuring this in each case.

\paragraph{\bf Type A, $4\nmid N$}
The domain of $G$ has odd order $N$ or $N/2$, at least seven because
$\varphi(N)=2g\ge6$. Suppose first that $3\nmid N$. If an odd prime
$q\ge7$ divides $N$, at least three choices for $c_q$ remain, and we can
choose a nonzero one. Then the CRT solution is nonzero. Otherwise
$N=5^e$ or $2\cdot5^e$, with $e\ge2$. If zero modulo $5$ is the only
allowed choice, it has a nonzero lift modulo $5^e$, for example $5$.
Thus any two nonzero residue classes have the same value of $G$.

If $3\mid N$, first let $a,b$ be nonzero multiples of $3$. Prescribe
$c\equiv1\pmod3$ and choose $c_q$ avoiding $\pm a,\pm b$ at every
other odd prime $q\mid N$. The resulting $c$ is nonzero, so all nonzero
multiples of $3$ have the same value. For $a\not\equiv0\pmod3$, prescribe
$c\equiv0\pmod3$ and avoid $\pm a$ at every other odd prime. If a prime
$q\ge5$ occurs, choose $c_q\ne0$ there as well. Otherwise $N=3^e$ or
$2\cdot3^e$, with $e\ge2$, and $c=3$ is nonzero in the domain of $G$.
In either case $G(a)=G(c)$. Hence $G$ is constant away from zero.

\paragraph{\bf Type A, $4\mid N$}
Suppose first that $3\nmid N$. For nonzero even $a,b$, choose
$c$ odd and $c_q$ avoiding $\pm a,\pm b$ at every odd prime $q\mid N$.
The odd parity makes $c$ nonzero modulo $N/2$, even if $c_5=0$ is forced.
Thus all nonzero even residue classes have the same value. For an odd
$a$, choose $c$ even and avoid $\pm a$ at every odd prime.
If an odd prime occurs, choose a nonzero allowed residue at one of them;
for a power of two, take $c=2$. Then $c\ne0$ and $G(a)=G(c)$, proving
constancy on all nonzero residue classes.

If $3\mid N$, first let $a,b$ be nonzero multiples of $3$ of the same
parity. Choose $c\equiv1\pmod3$ of the opposite parity, and avoid
$\pm a,\pm b$ at every prime $q\ge5$ dividing $N$. Thus $G(a)=G(c)=G(b)$.
For a nonmultiple $a$ of $3$, prescribe $c\equiv0\pmod3$ of opposite
parity and avoid $\pm a$ at every prime $q\ge5$ dividing $N$. A nonzero
allowed coordinate at one such prime keeps $c$ nonzero. If there is no
such prime, only $2,3$ divide $N$; since $N\ge24$, choose $c=3$ or $6$
according to the required parity. Both are nonzero modulo $N/2$.

It follows that $G$ has one value on
\[
 \{\text{even multiples of }3\}\ \cup\
 \{\text{odd nonmultiples of }3\}
       =\{a:a\equiv0,1,5\pmod6\}
\]
and one value on
\[
 \{\text{odd multiples of }3\}\ \cup\
 \{\text{even nonmultiples of }3\}
       =\{a:a\equiv2,3,4\pmod6\},
\]
with zero omitted. Writing these values as $C+E,C-E$ proves (i).

\paragraph{\bf Type B}
Here $4\mid N$ and $\varphi(N)=4g\ge12$, so $N\ge28$. We must also
avoid $N/4\pmod{N/2}$.

If $3\nmid N$ and $8\mid N$, then both excluded residue classes are even.
For two admissible even $a,b$, choose $c$ odd and avoid $\pm a,\pm b$
at every odd prime. It is automatically admissible. For an odd $a$,
choose $c$ even and avoid $\pm a$ at every odd prime. If such a prime
$q\ge5$ occurs, choose $c_q\ne0$ there; both excluded residue classes have
zero coordinate modulo $q$. If $N$ is a power of two, then $N\ge32$
and $c=2$ is admissible. Thus $G$ is constant on the admissible residue classes.

If $3\nmid N$ and $N\equiv4\pmod8$, then $N/4$ is odd. For two
admissible odd $a,b$, choose an even $c$ avoiding $\pm a,\pm b$ at every
odd prime. Since $N/4\ge7$ and $3\nmid N$, its odd prime factors include
a prime at least seven, or $N/4=5^e$ with $e\ge2$. As in the odd-order
case, a nonzero local coordinate or a nonzero lift modulo $5^e$ ensures
$c\ne0$. Its parity excludes $N/4$. Every admissible even $a$ can then
be compared with an odd $c$ avoiding $\pm a$ at all odd primes, and
having a nonzero coordinate at one of them. This excludes $N/4$ as well.
Again $G$ is constant on the admissible residue classes.

Finally suppose $3\mid N$. Both excluded residue classes are multiples of
$3$. For admissible multiples $a,b$ of $3$ of the same parity, choose
$c\equiv1\pmod3$ of opposite parity and avoid $\pm a,\pm b$ at every
prime $q\ge5$ dividing $N$. Such $c$ is automatically admissible. For
$a\not\equiv0\pmod3$, choose $c\equiv0\pmod3$ of opposite parity,
avoiding $\pm a$ at every prime $q\ge5$ dividing $N$. A nonzero allowed
coordinate at any such prime excludes both $0,N/4$. If no such prime
exists, only $2,3$ divide $N$; here $N\ge36$, so either $c=3$ or $6$
is admissible, according to the required parity. Thus $G(a)=C+Ef(a)$
on the admissible residue classes, with the same two parity classes as above.

It remains to apply \eqref{eq:B-antiperiod}. Translation by $N/4$
permutes the admissible residue classes. If $3\nmid N$, the single constant
equals its negative, so it is zero. If $24\mid N$, then
$f(a+N/4)=f(a)$, so each of the two constants equals its negative and
$C=E=0$. If $N\equiv12\pmod{24}$, then $f(a+N/4)=-f(a)$, so $C=0$.
This proves (ii).
\end{proof}

\subsection{Support bounds from Fourier inversion}

Now the Fourier inversion gives simple formulae for the coefficients $\lambda_{\ell}$. Suppose that $G(a)$ is constant for nonzero $a$. Then 

\[
 0=\lambda_0=
 \begin{cases}
 (\Lambda+(N-1)C)/N,&N\text{ odd},\\
 (2/N)(\Lambda+(N/2-1)C),&N\text{ even}.
 \end{cases}
\]
and 
\begin{equation}\label{eq:uniform-positive}
 \lambda_\ell=
 \begin{cases}
 \Lambda/(N-1)>0,&N\text{ odd},\\
 \Lambda/(N/2-1)>0,&N\text{ even},
 \end{cases}
 \qquad\ell\ne0.
\end{equation}

Equality in the support count gives one node pair per supported inversion
class. A nonfixed class has coefficient $D_i/2$ at each of its two
positions, whereas a fixed class has coefficient $D_i$ at one position.
Thus uniform coefficients give equal $D_i$ on the nonfixed classes and
half that weight on a fixed class.

If $12\mid N$, we can write
\[
 f(a)=\frac43\cos(\pi a/3)-\frac13(-1)^a.
\]
The inverse Fourier transform of $f$ on $\Z/(N/2)\Z$ is
\[
 \frac2N\sum_{a\bmod N/2}f(a)e^{-4\pi i a\ell/N}
 =\frac23(\delta_{\ell,N/12}+\delta_{\ell,5N/12})
      -\frac13\delta_{\ell,N/4}.
\]
The positions $\{N/12,5N/12\}$ and $\{N/4\}$ form two inversion
classes modulo $N/2$. In the exceptional cases below, every possibly
nonzero coefficient has value $t$, $t+2E/3$, or $t-E/3$, with the latter
two values at these two classes. If there is a nonexceptional position,
nonnegativity gives $t\ge0$. If $t=0$, nonnegativity at $N/12$ and $N/4$
forces both $E\ge0$ and $E\le0$. This would make all coefficients zero,
contrary to $\Lambda>0$. Thus $t>0$, and every nonexceptional inversion
class occurs. The two exceptional classes cannot both disappear: that
would require $t+2E/3=t-E/3=0$, contradicting $t>0$.

\begin{proposition}\label{prop:A-reduction}
In Type A, the only possibilities are $N=p$, $N=2p$ with
$p=2g+1$ odd prime, and $N=2^k$. If $N=p$, the node preimages are
exactly the nontrivial $p$-th roots of unity. In both prime cases,
all $D_i$ are equal. If $N=2p$ or $N=2^k$, there is exactly one
node pair in each distinct set
$\{\zeta_N^a,\zeta_N^{-a},-\zeta_N^{a},-\zeta_N^{-a}\}$ with
$a\not\equiv0\pmod{N/2}$. For $N=2^k$, the $D_i$ are equal except
that the pair $\{i,-i\}$ has half the common weight.
\end{proposition}

\begin{proof}
If $N$ is odd, then $G$ is constant except for $a=0$ by Lemma~\ref{lem:constancy}. So \eqref{eq:uniform-positive} holds and there are exactly $N-1$ nonzero $\lambda_{\ell}$. So we have $(N-1)/2\le g$. Hence $\varphi(N)=2g \ge N-1$ and $N=2g+1$ is prime. 
The $2g=p-1$ distinct node preimages therefore exhaust the nontrivial
$p$-th roots, which are all primitive.

If $N\equiv2\pmod4$, the same argument on $\ell\in \Z/(N/2)\Z$ gives
\[
              g\ge\frac{N/2-1}{2},\qquad
                          \varphi(N)=\varphi(N/2),
\]
so $p=N/2=2g+1$ is prime. Equality in the support count gives exactly
one node pair above each nonzero inversion class
$\{\ell,-\ell\}\subset\Z/p\Z$. These classes are nonfixed, so the
weight observation above and \eqref{eq:uniform-positive} give
$D_i=2\Lambda/(p-1)$ for every $i$. The same calculation applies to
the already established case $N=p$.

Suppose $4\mid N$ and $3\nmid N$. Then $G$ is constant except for
$a=0$ by Lemma~\ref{lem:constancy}. There are $N/4$ nonzero inversion
classes modulo $N/2$, including the fixed class $N/4$ corresponding
to $\{i,-i\}$. Thus we have
\[
                        g\ge N/4,\qquad
                     g=\varphi(N)/2\le N/4.
\]
Equality holds, so $N=4g$ is a power of two and exactly one node pair
lies above each nonzero inversion class modulo $N/2$. The same weight
observation gives half the common weight to the fixed pair $\{i,-i\}$.

Finally assume that $12\mid N$. By Lemma~\ref{lem:constancy},
$G(a)=C+Ef(a)$ at nonzero residue classes. Using the transform of $f$
computed above and accounting separately for $G(0)=\Lambda$, Fourier inversion gives,
for $\ell\ne0$,
\begin{equation*}
 \lambda_\ell=t+\frac{2E}{3}
           (\delta_{\ell,N/12}+\delta_{\ell,5N/12})
                                        -\frac E3\delta_{\ell,N/4}.
\end{equation*}
Here $t=2(\Lambda-C-E)/N=-C$, where the last equality follows from
$\lambda_0=0$.

There are nonexceptional positions, since $N\ge24$. The observation
above gives $t>0$ and at most one absent exceptional inversion class.
Thus at least $N/4-1$ of the $N/4$ nonzero inversion classes are
supported, giving
\[
                       g\ge N/4-1,
                     \qquad g=\varphi(N)/2\le N/6.
\]
These inequalities force $N\le12$, contrary to $12\mid N$ and $g\ge3$.
\end{proof}

\begin{proposition}\label{prop:B-reduction}
In Type B, the only remaining possibility is $N=2^k$. In that case
$g=N/8$, all $D_i$ are equal, and there is exactly one pair of nodes from each set $\{\zeta_N^a,-\zeta_N^a, \zeta_N^{-a},-\zeta_N^{-a} \}$ with odd $a$.
\end{proposition}

\begin{proof}
Here $4\mid N$ and $g\ge3$ implies $N\ge 28$. If $3\nmid N$, by Lemma~\ref{lem:constancy}, $G(a)=0$ for any admissible
$a\ne0,N/4$. The remaining values are $G(0)=\Lambda$ and, by
\eqref{eq:B-antiperiod}, $G(N/4)=-\Lambda$.  The Fourier inversion gives
\begin{equation}\label{eq:uniform-odd}
                  \lambda_\ell=\frac{2\Lambda}{N}
                                            (1-(-1)^\ell).
\end{equation} This is non-vanishing for odd $\ell$. 

If $8\mid N$, this has $N/4$ positive coefficients, grouped into
$N/8$ inversion classes. Thus
$g=\varphi(N)/4\ge N/8$, possible only when the equalities hold and $N=2^k$. 
The equality implies that each pair of nodes comes from one odd
inversion class. There is no fixed class, so the weights $D_i$ are equal.

If $N\equiv4\pmod8$, the odd class $N/4$ is fixed by inversion,
so the number of odd inversion classes is $(N+4)/8$, and we have $g=\varphi(N)/4\ge(N+4)/8>N/8$,  impossible because $\varphi(N)\le N/2$. 

Suppose $3\mid N$.  If $24\mid N$, then Lemma~\ref{lem:constancy}
again gives $G(a)=0$ for admissible $a\ne0,N/4$. The Formula~\eqref{eq:uniform-odd}
applies and gives $g\ge N/8$, contradiction to $g=\varphi(N)/4\le N/12$.

The remaining case is $N\equiv12\pmod{24}$, with $N\ge36$. By Lemma~\ref{lem:constancy}, $G(a)=Ef(a)$ for admissible $a\ne0,N/4$. Fourier inversion gives $\lambda_\ell=0$ for even $\ell$ and
\begin{equation*}
 \lambda_\ell=t+\frac{2E}{3}
       (\delta_{\ell,N/12}+\delta_{\ell,5N/12})
                -\frac E3\delta_{\ell,N/4},
       \qquad t=\frac{4(\Lambda-E)}N
\end{equation*} for odd $\ell$. 

There exist odd positions away from $N/12,5N/12$ and $N/4$. The same
observation gives $t>0$ and at most one absent exceptional inversion
class. Hence
\[
                    g\ge\frac{N+4}{8}-1=\frac{N-4}{8},
                       \qquad g=\varphi(N)/4\le N/12.
\]
This implies $N\le12$, a contradiction to $g\ge 3$. 
\end{proof}

\subsection{Determination of the node pairs}

For even $N$, the coefficients $\lambda_\ell$ determine only the sums
$\mu_\ell+\mu_{\ell+N/2}$. In order to determine which pair among $\{\zeta_N^{\pm \ell}, -\zeta_N^{\pm \ell}\}$ are actually chosen, we need to compute $\mu_\ell$. We now use (3) of Lemma~\ref{lem:top-duality},
$c_i=B_0D_i(x_i^s-x_i^{-s})$ for a unit $s$, which is odd. Substitution
into \eqref{eq:moments} and the product-to-sum identity give
\begin{equation}\label{eq:fourier-moments}
 F(s+r)=F(s-r),\qquad 1\le r\le g-1.
\end{equation}
Since the $D_i$ are rational, for every unit $u$ and every residue class $a$,
\[
 \tau_u(F(a))=F(ua),\qquad F(-a)=F(a).
\]
Since $x_i\ne\pm1$, we have $\mu_0=\mu_{N/2}=0$. The finite geometric
sum over odd residue classes gives
\[
 \sum_{\substack{a\bmod N\\a\text{ odd}}}F(a)
 =\sum_{\ell\bmod N}\mu_\ell
       \sum_{\substack{a\bmod N\\a\text{ odd}}}\zeta_N^{a\ell}
 =\frac N2(\mu_0-\mu_{N/2})=0.
\]
Indeed, the inner sum is zero unless $\ell=0,N/2$, where the sum is
$N/2,-N/2$, respectively. Fourier inversion also gives
\[
 \begin{aligned}
 \mu_\ell-\mu_{\ell+N/2}
 &=\frac1N\sum_{a\bmod N}F(a)\zeta_N^{-a\ell}(1-(-1)^a)\\
 &=\frac2N\sum_{\substack{a\bmod N\\a\text{ odd}}}
                       F(a)\zeta_N^{-a\ell}.
 \end{aligned}
\]

\begin{proposition}\label{prop:no-powers-two}
Neither Type A nor Type B with $N=2^k$ satisfies all the cusp conditions.
\end{proposition}
\begin{proof}
Since $g\ge3$ and $2^{k-1}=\varphi(N)=[L:\Q]\ge2g\ge6$, we have $k\ge4$.
Since $g-1\ge2$, taking $r=2$ in \eqref{eq:fourier-moments} gives
$F(s+2)=F(s-2)$. Both $s+2$ and $s-2$ are units $\pmod N$.
Applying $\tau_{(s-2)^{-1}}$ and then $\tau_a$ for any odd $a$ gives
\[
 F(ah)=F(a),\qquad
 h=(s+2)(s-2)^{-1}=1+4(s-2)^{-1}\equiv5\pmod8.
\]
For an integer representative of $h$, we use induction to prove
\[
 v_2(h^{2^r}-1)=r+2,\qquad r\ge0.
\]
The base case $v_2(h-1)=2$ is trivial, and we have $h^{2^{r+1}}-1= (h^{2^r}+1)(h^{2^r}-1)$ with $h^{2^r}+1\equiv2\pmod4$. Thus $h$ has order $2^{k-2}$ $\pmod N$ and generates the units congruent to $1\pmod4$.
So $h,$ and $-1$ generates $(\Z/N\Z)^\times$. Since $F(a)=F(ah)=F(-a)$ for odd $a$, we have $F(a)=C$ on all odd residue classes $a$.

The odd-class sum above gives $(N/2)C=0$, so $F(a)=0$ for every odd $a$.
The difference formula therefore gives
\[
 \mu_\ell=\mu_{\ell+N/2}.
\]
There is a selected node $x_i=\zeta_N^{n_i}$ with $n_i$ odd, since $N$
is the actual least common multiple of the node orders. Then $x_i$ has
order $N$, and $\mu_{n_i}>0$ forces also $\mu_{n_i+N/2}>0$. Thus both
$x_i^{\pm}$ and $-x_i^{\pm}$ are node branches. So they have to be the same pair by Propositions~\ref{prop:A-reduction} and~\ref{prop:B-reduction}. This is possible only if $x_i^2=-1$ and $n_i\equiv N/4\pmod N$ with odd $n_i$, contradiction to $N\ge16$. 
\end{proof}

\begin{proposition}\label{prop:prime-signs}
Suppose $N=2p$, where $p=2g+1$ is prime. Then the node preimages
$\{x_i,x_i^{-1}:1\le i\le g\}$ are the complete set of primitive
$2p$-th roots.
\end{proposition}
\begin{proof}
We identify $\Gal(K/\Q)=(\Z/2p\Z)^\times/\{\pm1\}\cong (\Z/p\Z)^\times/\{\pm1\}$ by reduction$\pmod p$, and choose
one representative for each inversion class. By Proposition~\ref{prop:A-reduction},
there is exactly one node pair from $\{\zeta_p^{\pm a},-\zeta_p^{\pm a}\}$.
After relabeling the pairs, we can write
\[
 x_a=\epsilon_a\zeta_p^a,\quad a\in \Gal(K/\Q),\qquad \epsilon_a\in\{1,-1\}.
\]
Reversing a pair also changes the sign of $c_a$, while $D_a$ and $F$
are unchanged. Let $t=\Lambda/(p-1)>0$. Proposition~\ref{prop:A-reduction}
and \eqref{eq:uniform-positive} give, for $\ell\not\equiv0,p\pmod{2p}$,
\[
 \mu_\ell+\mu_{\ell+p}=t>0,\qquad \mu_\ell\mu_{\ell+p}=0.
\]
Also $\mu_0=\mu_p=0$, and we have 
\[
 F(0)=\Lambda=(p-1)t,\qquad F(2a)=-t,\quad a\not\equiv0\pmod p.
\]

We first determine $s\pmod p$ that satisfies (3) of Lemma~\ref{lem:top-duality}. Suppose that $s\equiv\pm r\pmod p$
for some $1\le r\le g-1$. If $r$ is odd, one of $s+r,s-r$ is zero
modulo $2p$ and the other is a nonzero even residue class.
Equation~\eqref{eq:fourier-moments} would give $\Lambda=-t$, which is
impossible. If $r$ is even, one of $s+r,s-r$ is $p$ and the other
is a unit $u$, so $F(p)=F(u)$. Applying all $\tau_v$ gives
$F(a)=F(p)$ for every unit $a$, since $vp\equiv p\pmod{2p}$.
Thus all odd residue classes have the same value. Their sum is zero,
so all these values vanish. The difference formula above gives
$\mu_\ell=\mu_{\ell+p}$, contrary to
$\mu_\ell+\mu_{\ell+p}=t>0$ and $\mu_\ell\mu_{\ell+p}=0$.
Therefore $s\equiv\pm g\pmod p$. Replacing $s$ by $-s$ if necessary,
we may take
\[
 s=g\quad(g\text{ odd}),\qquad s=g+1\quad(g\text{ even}).
\]

For even $r$ with $2\le r\le g-1$, both $s+r$ and $s-r$ are units$\pmod {2p}$. Applying Galois automorphisms to
\eqref{eq:fourier-moments} gives
\[
 F\bigl(u(s+r)(s-r)^{-1}\bigr)=F(u)
 ,\qquad u\in(\Z/2p\Z)^\times.
\]
As in Proposition~\ref{prop:no-powers-two}, we want to prove $(s+r)(s-r)^{-1}$ generates $\Gal(K/\Q)$. Let $J\le \Gal(K/\Q)$ be the subgroup generated by $(s+r)(s-r)^{-1}$ for even $2\leq r\leq g-1$. In the following discussion, ratios are taken in $\Gal(K/\Q)=(\Z/p\Z)^\times/\{\pm1\}$.

If $g$ is odd, then $s=g$ and $2s=p-1$. For even $r$, the elements $n=s+r$ run through $1,3,\ldots,p-2$. Since
\[
 s-r=2s-n\equiv-(n+1)\pmod p,
\]
we obtain 
\[
 \frac12,\ \frac34,\ \frac56,\ldots,\frac{p-2}{p-1}\in J.
\]
In particular, $2\in J$.

If $g$ is even, then $s=g+1$ and $2s=p+1$. The allowed even $r$,
again with both signs and zero, give $n=s+r=3,5,\ldots,p-2$. Since
\[
 s-r=2s-n\equiv-(n-1)\pmod p,
\]
we obtain
\[
 \frac32,\ \frac54,\ \frac76,\ldots,\frac{p-2}{p-3}\in J.
\]
Here $p\ge13$. In particular, $3/2,9/8\in J$, and therefore
\[
 2=\left(\frac32\right)^2\left(\frac98\right)^{-1}\in J.
\]

In either case, we use induction on $n$ to prove all $1\le n\le p-1$ belongs to $J$. If $n$ is even, then $n=2(n/2)\in J$
by induction. If $n>1$ is odd, use the identity
\[
 n=\frac{n}{n+1}\,2\left(\frac{n+1}{2}\right),\quad g\text{ odd},
 \qquad
 n=\frac{n}{n-1}\,2\left(\frac{n-1}{2}\right), \quad g\text{ even}.
\]
The ratio and $2$ is in $J$, and the last factor is a positive integer smaller
than $n$. Hence $n\in J$ by induction. 

Now again consider $\Gal(K/\Q)=(\Z/2p\Z)^\times/\{\pm1\}$. Since $F(-u)=F(u)$ and by above discussion, we conclude that $F(u)=C$ for some constant $C$ on units. The only odd residue class that is not a unit is
$p$, so the odd-class sum gives $F(p)=-(p-1)C$.
For $\ell\not\equiv0,p\pmod{2p}$, summing over all odd residue classes
and removing the term at $p$ gives
\[
 \sum_{u\in(\Z/2p\Z)^\times}\zeta_{2p}^{-u\ell}=-(-1)^\ell.
\]
The difference formula therefore yields
\[
 \begin{aligned}
 \mu_\ell-\mu_{\ell+p}
 &=\frac1p\left(C\sum_{u\in(\Z/2p\Z)^\times}
          \zeta_{2p}^{-u\ell}+F(p)(-1)^\ell\right)\\
 &=-C(-1)^\ell.
 \end{aligned}
\]
Together with $\mu_\ell+\mu_{\ell+p}=t$, we have
\[
 \mu_\ell=\frac{t-C(-1)^\ell}{2},\qquad
 0=\mu_\ell\mu_{\ell+p}=\frac{t^2-C^2}{4}.
\]
Thus $C=\pm t$. If $C=-t$, only $\mu_\ell$ for nonzero even $\ell$ have positive coefficients, so every node has order $p$, contradiction to $N=\operatorname{lcm}\operatorname{ord}x_i=2p$.
Hence $C=t$, and the supported exponents are exactly the odd classes
other than $p$, namely the units modulo $2p$. Equivalently, all
$x_a=-\zeta_p^a$, and the node pairs form the complete set of
primitive $2p$-th roots of unity.
\end{proof}

\begin{proof}[Proof of Theorem~\ref{thm:cusp-rigidity}]
By Propositions~\ref{prop:A-reduction}, \ref{prop:B-reduction} and
\ref{prop:no-powers-two}, only Type A remains, with $p=2g+1$ prime
and $N=p$ or $2p$. For $N=p$, the node preimages are already all
primitive by Proposition~\ref{prop:A-reduction}. For $N=2p$, this follows
from Proposition~\ref{prop:prime-signs}. Thus in either case they are
exactly the $2g$ primitive $N$-th roots. The equality of the weights
$D_i$ follows from Proposition~\ref{prop:A-reduction}. Since $K=L^+$ and
$\Q(\zeta_{2p})=\Q(\zeta_p)$, we have $K=\Q(\zeta_p)^+$.
The product in Proposition~\ref{prop:MM} is $\Phi_N(z)$, so
\[
 \omega_h^{\mathrm{norm}}=R\frac{z^{g-1}}{\Phi_N(z)}\,dz.
\]
For $\phi(z)=-z$, the identity $\Phi_{2p}(z)=\Phi_p(-z)$ gives
\[
 \phi^*\omega_p=(-1)^g\frac{z^{g-1}}{\Phi_{2p}(z)}\,dz.
\]
Moreover, $\phi$ sends each primitive $p$-pair to a primitive $2p$-pair
and preserves inversion pairing. Thus the two presentations define
isomorphic projective stable differentials, represented by
\eqref{eq:polygon-cusp}.
\end{proof}

\section{Uniqueness of the Veech polygons}\label{sec:uniqueness}

In the preceding section, we proved uniqueness of the paired projective
stable differential at an irreducible cusp of an algebraically primitive Teichm\"uller curve. The regular $2p$-gon with opposite sides identified generates such a curve, with trace field $\Q(\zeta_p)^+$. Its horizontal direction through opposite vertices
is irreducible. By Theorem~\ref{thm:cusp-rigidity}, any arbitrary algebraically primitive Teichm\"uller curve has the same irreducible cusp. By \cite[Lemma~2.5]{Winsor}, there exists a generating differential $(X,\omega)$ of this curve with the same heights, circumferences of horizontal cylinders of the regular $2p$-gon. From the regular $2p$-gon, when $\theta=\pi/p$,
\[
c_i=\ell_{i-1}+\ell_{i},\qquad
h_i=\frac12\tan(\theta/2)c_i,
\quad1\le i\le g
\] where 
$\ell_i=2\cos(i\theta)$.

We label horizontal cylinder $C_i$ consecutively along the chain $C_1 - C_2 -\dots-C_g.$ Then there is a permutation 
$\pi\in \operatorname{Sym}_g$ such that, $c_{\pi(i)}, h_{\pi(i)}$ are their circumferences and heights.

In order to prove uniqueness of the curve, now we take care of the permutation $\pi$ and the twists, i.e. how we glue these cylinders $C_i$. Let $t_i\in\R/c_{\pi(i)}\Z$ be the horizontal displacement from a marked
zero on the bottom boundary of $C_i$ to the same zero on its top
boundary, as in \cite[\S4]{Winsor}. Use the positive horizontal
orientation on both boundary circles. With these markings, $t_i=0$
means that the corresponding boundary intervals align vertically.
A horizontal shear with parameter $s$ changes $t_i$ to
$t_i+sh_{\pi(i)}\pmod {c_{\pi(i)}}$, without changing any circumference, height,
or horizontal boundary length.

\begin{proposition}\label{prop:reconstruction}
For any algebraically primitive Veech surface with the data
above, we have $\pi=\operatorname{id}$ and a common horizontal shear makes
\[
t_1=\cdots=t_g=0.
\]
Consequently, any such surface belongs to the $\GL^+(2,\R)$-orbit of the regular $2p$-gon.
\end{proposition}

\begin{proof}
In the case of the regular $2p$-gon, we have $t^{\text{pol}}_i=\frac{c_i}{2}$ for each $i=1,\dots,g$ and a shear \[\begin{pmatrix}
1& - \cot(\theta/2)\\
0&1
\end{pmatrix}=
\begin{pmatrix}
1& - c_i/2 h_i\\
0&1
\end{pmatrix}
\] makes all twists vanish.

So we also apply the horizontal shear
\[
\begin{pmatrix}
1&-t_1/h_{\pi(1)}\\
0&1
\end{pmatrix}.
\]
This makes $t_1=0$ and preserves all the dimensions fixed above.
The self-glued interval of length $\ell$ at $C_1$ then suspends
to a vertical cylinder $E_1$, with circumference $h_{\pi(1)}$ and transverse
width $\ell$. Its endpoint trajectories are saddle connections
from each marked zero to itself, so the cylinder is maximal.
By Veech dichotomy \cite{Veech}, the vertical direction is completely periodic,
and the ratios of its cylinder moduli are rational. Since this direction contains a saddle connection that is a loop, every saddle connection in it is a loop \cite[Remark~2.6]{Winsor}. 
The degeneration in this direction thus has a reducible cusp with two rational components, one for each zero by \cite[Theorem~2.1 and Corollary~2.2]{Mol08}. The genus formula gives $g+1$ nodes between two component, corresponding to $g+1$ vertical cylinders, labeled as $E_j$,
$j=1,\ldots,g+1$. Let $v_j>0$ be the circumference of $E_j$ and $b_j>0$ its height. The ratio $v_j/b_j$ are all rational. 

Let $n_{ij}\in\Z_{\ge0}$ be the number of times a core curve of
$E_j$ crosses $C_i$, for $1\le i\le g$ and $1\le j\le g+1$.
Each crossing traverses the height $h_{\pi(i)}$. Conversely, a horizontal
core of $C_i$ meets $E_j$ in $n_{ij}$ intervals, each of length $b_j$.
These intervals partition the core away from finitely many boundary
points. Therefore
\begin{equation}\label{eq:intersection-rectangles}
v_j=\sum_{i=1}^g n_{ij}h_{\pi(i)},
\qquad
c_{\pi(i)}=\sum_{j=1}^{g+1} n_{ij}b_j.
\end{equation}
For the cylinder $E_1$, 
\[
n_{i1}=\delta_{i1},\qquad
v_1=h_{\pi(1)},\qquad b_1=\ell.
\]

Define the positive rational ratios
\[
r_j=\frac{b_j/v_j}{b_1/v_1}
    =\frac{b_jh_{\pi(1)}}{2v_j}\in\Q_{>0}.
\]
Then $r_1=1$ and
\[
b_j=\frac{2}{h_{\pi(1)}}r_jv_j.
\]
Define the matrix $A$ by
\[
    A_{ik}:=\sum_{j=1}^{g+1} r_jn_{ij}n_{kj}.
\]
It is a rational symmetric $g\times g$ matrix with nonnegative entries. It satisfies
\[
    \operatorname{tr}(A)=\sum_{j=1}^{g+1} = r_jn_{ij}^2>0.
\] 
By substituting into \eqref{eq:intersection-rectangles}, and using
$h_{\pi(k)}/h_{\pi(1)}=c_{\pi(k)}/c_{\pi(1)}$, we have
\[
 c_{\pi(i)}=\frac{2}{h_{\pi(1)}}\sum_{j=1}^{g+1}\sum_{k=1}^g
       r_jn_{ij}n_{kj}h_{\pi(k)}
     =\frac{2}{c_{\pi(1)}}\sum_{k=1}^g A_{ik}c_{\pi(k)},
\]
that is,
\[
Ac_{\pi}=\frac{c_{\pi(1)}}{2}c_{\pi},
\qquad
c_{\pi}=(c_{\pi(1)},\ldots,c_{\pi(g)})^{\mathsf T}.
\]

The circumferences $c_1,\ldots,c_g$ form a basis of $K$ over $\Q$, so the $i$th row of $A$ is the unique coordinate vector of $(c_{\pi(1)}/2)c_{\pi(i)}$ in this (permuted) basis. 
Since
\[
c_i=2\cos((i-1)\theta)-2\cos(i\theta),
\qquad 1\le i\le g,
\]
We have
\[
    \operatorname{Tr}_{K/\Q}\left(\frac{c_i}{2}\right)=\operatorname{Tr}_{K/\Q}(\cos((i-1)\theta))-\operatorname{Tr}_{K/\Q}(\cos(i\theta)).
\]

A simple computation gives \[\operatorname{Tr}_{K/\Q}(\cos(i\theta))=\frac{(-1)^{i+1}}{2},\qquad i=1,\dots, g\] and $\operatorname{Tr}_{K/\Q}(1)=g$. Thus we must have $\pi(1)=1$ in order to have $\operatorname{tr}(A)>0$. 

Now with $Ac_{\pi}=\frac{c_1}{2}c_{\pi}$, the product-to-sum identity gives

\[
 \frac{c_1}{2}c_1=\frac32c_1+\frac12c_2,
 \qquad
 \frac{c_1}{2}c_i=\frac12c_{i-1}+c_i+\frac12c_{i+1}
 \quad(2\le i<g),
 \qquad
 \frac{c_1}{2}c_g=\frac12c_{g-1}+c_g
\]
Since $C_1$ and $C_2$ are adjacent, we must have $A_{12}>0$, and this is only possible if $\pi(2)=2$. By repeating this, we can conclude that $\pi=\operatorname{id}$. Also we have $\ell=\ell_0$ and the boundary interval between $C_i$ and $C_{i+1}$ has length $\ell_{i}$. 

So we obtain
\begin{equation}\label{eq:jacobi-matrix}
A_{ik}=
\begin{cases}
3/2,&i=k=1,\\
1,&i=k\ge2,\\
1/2,&|i-k|=1,\\
0,&|i-k|\ge2.
\end{cases}
\end{equation}

For $|i-k|\ge2$, the zero entry in \eqref{eq:jacobi-matrix} gives
\[
0=A_{ik}=\sum_{j=1}^{g+1} r_jn_{ij}n_{kj}.
\]
Each summand is nonnegative and every $r_j$ is positive. Thus no
vertical cylinder $E_j$ meets two nonadjacent horizontal cylinders.

Fix $2\le i\le g-1$. A regular vertical trajectory entering $C_i$
from $C_{i-1}$ must exit toward $C_{i-1}$: otherwise its cylinder
would meet both $C_{i-1}$ and $C_{i+1}$. The incoming and outgoing
intervals toward $C_{i-1}$ both have length $\ell_{i-1}$. Except for
finitely many points on vertical saddle connections, translation
across $C_i$ maps the incoming interval into the outgoing one.
Their lengths are equal, so these intervals agree after translation,
up to endpoints. In the marked boundary coordinates they occupy the
same proper interval when $t_i=0$. A nonempty proper circle interval
is preserved by a translation only when that translation is zero.
Hence $t_2=\cdots=t_{g-1}=0$.

It remains to prove $t_g=0$. A regular periodic trajectory entering
$C_g$ from $C_{g-1}$ eventually returns to $C_{g-1}$. We must rule
out extra self-returns in $C_g$ between consecutive visits to
$C_{g-1}$.

For $1\le i\le g-2$, the cylinders $C_i$ and $C_{i+1}$ now both
have zero twist. Their matching boundary intervals of length $\ell_i$
form a vertical cylinder crossing each of them once and no other
horizontal cylinder. Relabel these cylinders as $E_{i+1}$.
Their transverse widths are $\ell_i$, their circumferences are
$h_i+h_{i+1}$, and their intersection vectors are $e_i+e_{i+1}$.
These cylinders have the same dimensions as the corresponding polygon
cylinders, so $r_{i+1}=1/2$.

The cylinders $E_2,\ldots,E_{g-1}$ fill those boundary intervals
completely. In particular, $E_{g-1}$ fills the part of $C_{g-1}$
adjacent to $C_{g-2}$. Among the remaining cylinders, set
\[
J=\{j\in\{g,g+1\}:n_{g-1,j}>0\}.
\]
Each $E_j$ with $j\in J$ also meets $C_g$, since $C_{g-1}$ does not have self-meeting boundary. Also, successive visits of a
periodic core to $C_{g-1}$ are separated by at least one visit to
$C_g$, so $n_{g,j}\ge n_{g-1,j}>0$. The cylinder $E_{g-1}$ contributes $1/2$ to $A_{g-1,g-1}$ and
nothing to $A_{g-1,g}$. Thus \eqref{eq:jacobi-matrix} gives
\[
A_{g-1,g-1}=\sum_{j\in J}r_jn_{g-1,j}^2=\frac12,
\qquad
A_{g-1,g}=\sum_{j\in J}r_jn_{g-1,j}n_{g,j}=\frac12.
\]
Subtracting two equations, we obtain
\[
\sum_{j\in J}r_jn_{g-1,j}\bigl(n_{g,j}-n_{g-1,j}\bigr)=0.
\]
Every summand is nonnegative and $r_j, n_{g-1,j}>0$ for $j\in J$, so $n_{g,j}=n_{g-1,j}$ for all $j\in J$.
There is therefore exactly one visit to $C_g$ between consecutive
visits to $C_{g-1}$. Apart from the finitely many singular trajectories,
the incoming interval of length $\ell_{g-1}$ in $C_g$ maps directly
to the outgoing interval toward $C_{g-1}$. The same argument as above gives $t_g=0$.

The marked cylinders now have the same twists as the sheared regular $2p$-gon. They therefore define equivalent translation surfaces. 
\end{proof}

\begin{proof}[Proof of Theorem~\ref{thm:classification}]
Choose an irreducible periodic direction as in Proposition~\ref{prop:MM}. By Theorem~\ref{thm:cusp-rigidity}, we have $p=2g+1$ prime and the irreducible cusp differential is equal to that of the regular $2p$-gon. Then Proposition~\ref{prop:reconstruction} proves that the smooth generating differential is in a $\GL^+(2,\R)$-orbit of regular $2p$-gon. \end{proof}

\end{document}